\documentclass[11pt]{article}

\usepackage{bigints}
\usepackage{IEEEtrantools}
\usepackage{amsfonts}
\usepackage[usenames,dvipsnames,svgnames,table]{xcolor} 
\usepackage{enumerate} 
\usepackage{graphicx}
\usepackage{mathtools}

\usepackage{mathrsfs}
\usepackage{tikz-cd}
\usepackage{amsmath, amsthm}
\usepackage{imakeidx}
\usepackage[bookmarksnumbered]{hyperref}

\theoremstyle{definition}
\newtheorem{defn}{Definition}[section]
\newtheorem{thm}[defn]{Theorem}

\newtheorem{prop}[defn]{Proposition}
\newtheorem{cor}[defn]{Corollary}

\newtheorem{rem}[defn]{Remark}

\theoremstyle{definition}
\newtheorem{defin}{Definition}

\newtheorem{lema}[defin]{Lemma}

\makeindex[columns=2, intoc]
\hypersetup{
    linktoc=all,
    colorlinks=True,
    linkcolor=Blue,
    citecolor=ProcessBlue,
}
\usepackage{graphicx} 

\title{On representations of the crystallization of the quantized function algebra $C(SU_{q}(n+1))$}
\author{Saikat Das\space and Ayan Dey}
\date{}
\begin{document}
\maketitle
\begin{abstract}
The crystal limit $C(K_{0})$ of the $q$-family of $C^{*}$-algebras $C(K_{q})$ was introduced by Giri \& Pal for all $K=SU(n+1),\, n\geq 2$. This article aims to prove that the crystal limits $C(K_{0})$ have the property that the representations of $C(K_{q})$ give rise to the representations of the crystallized algebra $C(K_{0})$ by sending generators of $C(K_{0})$ to the limit of (scaled) generators of $C(K_{q})$ and every representation of $C(K_{0})$ occurs in this way. This work addresses a question raised by Giri $\&$ Pal in $\cite{GirPal-2024}$. As a consequence, one can realize $C(K_{0})$ as the $C^{*}$-algebra generated by the limit operators of faithful representations of $C(K_{q})$. 
\end{abstract}
{\bf AMS Subject Classification No.:}
20G42, 
46L67, 
58B32 
\section{Introduction}
The crystallization of the $C^{*}$-algebra $C(K_{q})$ was introduced by Giri \& Pal in \cite{GirPal-2023} for type $A_{n}$ compact Lie groups $K=SU(n+1)$,\, $n\geq 2$ and subsequently by Matassa \& Yuncken in \cite{MatYun-2023} for all simple simply connected compact Lie groups $K$. These two definitions of crystallization are slightly different but it is expected that the two crystallizations give the same $C^{*}$-algebra. However, we will only work with the definition of crystallization given in \cite{GirPal-2023} where the authors had observed that for every irreducible representation of $C(K_{q})$, limits of the canonical generating elements, scaled appropriately, obey the same set of relations as $q\to 0^{+}$ and used
those relations to define the crystallized algebra. Throughout this article, we will denote $C(SU_{q}(n+1))$ by $A_{q}(n)$ for any $q\in [0,1)$. Furthermore, to make the notation more convenient, we may sometimes omit $n$ and write $A_{q}$ for $A_{q}(n)$. In \cite{GirPal-2024} Giri \& Pal studied irreducible representations of the crystallized algebra $A_{0}(n)$ in great detail. In particular, they classified all the irreducible representations of the crystallized algebra and proved that every irreducible representation of $A_{0}(n)$ is the norm limit of irreducible representations of $A_{q}(n)$ with appropriately scaled generators. The goal of this article is to extend this result for any arbitrary non-degenerate (respectively faithful) representation of $A_{0}(n)$ and, conversely, any non-degenerate (respectively faithful) representation of $A_{q}(n)$ gives rise to a non-degenerate (respectively faithful) representation of $A_{0}(n)$ in a natural way (see Theorems \ref{main1}, \ref{main2}).\\
Now we will say a few words about the content of this paper. In section 2, the representation theory of $A_{0}(n)$ is recalled very briefly, and some of the important results of \cite{GirPal-2024} are listed. We conclude this section by proving that the definition of irreducible representations of $A_{0}(n)$ is independent of the choice of a reduced form taken for a permutation $w\in S_{n+1}$. Section 3 is devoted to understanding the topology and Borel structure of the spectrum of $A_{q}(n)$. The crucial observation is that the spectrums $\hat{A}_{0}$ and $\hat{A}_{q}$ are Borel isomorphic via a specific map and a sufficient condition for factorization of irreducibles is described in terms of Bruhut ordering of $\mathcal{S}_{n+1}$. The initial part of section 4 is all about decomposing any arbitrary non-degenerate representation of $A_{q}(n)$ into multiple copies of a direct integral of irreducibles using a result from \cite{Dix-1977} and, finally we then come to the main results of this paper that any non-degenerate representation (respectively faithful) of $A_{0}(n)$ on a separable Hilbert space $\mathcal{H}$ is the norm limit of non-degenerate representations (respectively faithful) of $A_{q}(n)$ on the same Hilbert space $\mathcal{H}$ with appropriately scaled generators, and the converse. The two crucial steps of the proof of faithfulness for different values of $q\in [0,1)$ are identifying the minimal faithful part of a faithful representation and capturing faithfulness in terms of support of the underlying positive Borel measure of direct integrals. As a consequence, we will then be able to realize the crystallized algebra $A_{0}(n)$ as the $C^{*}$-algebra generated by the limit operators of faithful representations of $A_{q}(n)$ with suitably scaled generators (see Corollary \ref{cor}). 
\section{Irreducible representations of $A_{0}(n)$}
For any $q\in(0,1)$, $A_{q}(n)$ is a well-studied universal $C^{*}$-algebra with $(n+1)^{2}$ many generators $z_{i,j}(q)$ satisfying certain relations (see $\cite{GirPal-2023}$). Given two representations $\pi_{1}$ and $\pi_{2}$ of $A_{q}(n)$, their convolution product is defined by the formula $\pi_{1}*\pi_{2}:= (\pi_{1}\otimes \pi_{2})\circ \Delta_{q}$ where $\Delta_{q}: A_{q}\to A_{q}\otimes A_{q}$ is the comultiplication map on $A_{q}$ associated to the quantum group structure of $A_{q}$. $\mathcal{S}_{n+1}$ denotes the permutation group on $n+1$ letters and the collection $\{s_{i}=(i,i+1): 1\leq i\leq n\}$ of Coxeter generators generates the group $\mathcal{S}_{n+1}$. We use the symbol $S$ for the left shift/backward shift operator on $l^{2}(\mathbb{N})$ and, $q^{N}$ is the operator on $l^{2}(\mathbb{N})$ that acts on the standard basis $\{e_{n}\}_{n\geq 0}$ by, $e_{n}\to q^{n}e_{n}$ for all $n\in \mathbb{N}\cup \{0\}$. Then for any $s_{r}\in \mathcal{S}_{n+1}$ one can define,
\begin{align*}
    \pi_{s_{r}}^{(q)}(z_{i,j}(q)):= \left\{ 
    \begin{matrix}
        S\sqrt{1-q^{2N}} & i=j=r \\
        \sqrt{1-q^{2N}}S^{*} & i=j=r+1\\
        -q^{N+1} & i=r, j=r+1\\
        q^{N} & i=r+1 , j=r\\
        \delta_{i,j}I & \text{otherwise}.
        \end{matrix}\right..
\end{align*}
For $a,b \in \mathbb{N}$ with $a\leq b$, let us denote by $s_{[a,b]}$
the product $s_{b}s_{b-1}\cdots s_{a}$ in $\mathcal{S}_{n+1}$. Let
$1\leq k\leq n,\;1\leq b_{k}< b_{k-1}< \cdots < b_{1}\leq n,$ and $1\leq a_{i} \leq b_{i}$ for $1 \leq i \leq k$. It follows from the
strong exchange condition and the deletion condition in the characterization of the Coxeter system (see $\cite{BjoBre-2004}$) that $ w = s_{[a_{k},b_{k}]}s_{[a_{k-1},b_{k-1}]}\cdots s_{[a_{1},b_{1}]}$ is a reduced word in $\mathcal{S}_{n+1}$.
For any reduced word $w$ in this form we define,
\begin{center}
    $\pi_{w}^{(q)}:= \pi_{s_{[a_{k},b_{k}]}}^{(q)}* \pi_{s_{[a_{k-1},b_{k-1}]}}^{(q)}*\cdots *\pi_{s_{[a_{1},b_{1}]}}^{(q)}$.
\end{center}
Denote the maximal torus of $SU(n+1)$ by $\mathbb{T}\cong T^{n}$. Then, for any $t\in \mathbb{T}$ we have all one-dimensional irreducibles of $A_{q}(n)$ defined as follows,
\begin{center}
    $\chi_{t}^{(q)}(z_{i,j}(q)):= \left\{ 
    \begin{matrix}
        t_{1}\delta_{i,j} & i=1\\
        \overline{t_{n}}\delta_{i,j} & i=n+1\\
        \overline{t_{i-1}}t_{i}\delta_{i,j} & \text{otherwise}
        \end{matrix}\right.$
\end{center}
\begin{thm}[\cite{KorSoi-1998}]
\label{muq}
    \textit{Any irreducible representation of $A_{q}(n)$ is of the form $\pi_{t,w}^{(q)}:=\chi_{t}^{(q)}*\pi_{w}^{(q)}$ for some $t\in \mathbb{T}$ and a reduced word $w\in \mathcal{S}_{n+1}$. Furthermore, $\pi_{t,w}^{(q)}$'s are all mutually inequivalent.}
\end{thm}
The crystallization $A_{0}(n)$ of the $q$-family $C^{*}$-algebras $A_{q}(n)$' is defined in \cite{GirPal-2023} as a universal $C^{*}$-algebra with a finite set of generators $\{z_{k,j}(0):1\leq k,j \leq n+1\}$ satisfying certain relations. One of the main advantages of this definition is that one can immediately conclude the following are representations of $A_{0}(n)$ given by the formula,
\begin{align*}
\label{1}
    \pi^{(0)}_{t,w}(z_{i,j}(0)):=\left\{\begin{matrix}\lim_{q\rightarrow 0^+} \pi^{(q)}_{{t,w}}(z_{i,j}(q))&\text{ if } i\geq j\\\lim_{q\rightarrow 0^+} (-q)^{i-j}~\pi^{(q)}_{{t,w}}(z_{i,j}(q))&\text{ if } i< j\end{matrix}\right.
\end{align*}
for $t\in\mathbb{T}$ and any reduced word $w\in \mathcal{S}_{n+1}$ in the above form. In particular, for $s_{r}\in \mathcal{S}_{n+1}$ we have
\begin{align*}
    \pi^{(0)}_{s_{r}}(z_{i,j}(0)):= \left\{ 
    \begin{matrix}
        S & i=j=r \\
        S^{*} & i=j=r+1\\
        P_{0} & i=r, j=r+1\\
        P_{0} & i=r+1 , j=r\\
        \delta_{i,j}I & \text{otherwise}
        \end{matrix}\right.
\end{align*} 
where $P_{0}=I-S^{*}S$ is the orthogonal projection onto the subspace of $l^{2}(\mathbb{N})$ generated by $e_{0}$.
\begin{rem}

Recall that even though the elements $z_{i,j}(q)$ belong to different $C^*$-algebras $A_{q}(n)$ for different values of $q$, they are
images of the generating elements $z_{i,j}(t)$ of the space $\mathcal{O}_{t}^{A}(SU(n+1))$ under the specialization maps $\theta_{q}$ (see Section 2, \cite{GirPal-2023} for details; there $z_{i,j}(q)$ and $z_{i,j}(t)$ are denoted by $u_{i,j}(q)$ and $u_{i,j}(t)$ respectively).
\end{rem}
\begin{thm}[\cite{GirPal-2024}]
\begin{enumerate}
\item \textit{$\pi^{(0)}_{t,w}$'s are all mutually inequivalent irreducible representations of the crystallized algebra $A_{0}(n)$.}
\item \textit{$A_{0}(n)$ is a separable $C^{*}$-algebra of type I i.e. the image of any irreducible representation on a Hilbert space $\mathcal{H}$ contains $K(\mathcal{H})$.} 
\end{enumerate}
\label{mu0}
\end{thm}
\noindent Altogether we have the following important and desirable property of the crystallized algebra.
\begin{cor}[\cite{GirPal-2024}]
\label{irr}
 \textit{Denote by 
$z_{k,j}^{}(q)$ the generators
of $A_{q}(n)$ for $q\in[0,1)$.
Let $\pi^{(0)}$ be an irreducible representation of $A_{0}(n)$
on a seperable Hilbert space $\mathcal{H}$. Then there exists irreducible
representations $\pi^{(q)}$ of $A_{q}$ 
on the same Hilbert space $\mathcal{H}$ such that
}
\begin{IEEEeqnarray*}{rCl}
\pi^{(0)}(z_{k,j}^{}(0))=
 \lim_{q\to 0+}
 \pi^{(q)}\left((-q)^{\min\{k-j,0\}}z_{k,j}^{}(q)\right),\qquad
  k,j\in\{1,2,\ldots,n+1\}.
\end{IEEEeqnarray*}
\end{cor}
\vskip 0.1cm
\noindent Let us recall that $A_{0}(n)$ can be equipped with natural coalgebra structure with the comultiplication map $\Delta:A_{0}(n)\to A_{0}(n)\otimes A_{0}(n)$ given by the formula
\begin{align*}
    \Delta(z_{i,j}(0))=\sum_{k=min\{i,j\}}^{max\{i,j\}} z_{i,k}(0)\otimes z_{k,j}(0)
\end{align*}
and using this formula for $\Delta$ one can then define the convolution of two representations of $A_{0}(n)$ as in the case for $q\neq 0$. Therefore, it makes sense to talk about the representation $\pi^{(0)}_{w}:= \pi^{(0)}_{s_{i_{1}}}*\pi^{(0)}_{s_{i_{2}}}*\cdots*\pi^{(0)}_{s_{i_{n}}}$ for any reduced word $w=s_{i_{1}}s_{i_{2}}\cdots s_{i_{n}}$. So, given $w\in \mathcal{S}_{n+1}$, there can be multiple reduced-form choices for $w$ and hence multiple representations each of which is defined for a specific choice of reduced form for $w$. Although the authors have described all the irreducible representations of $A_{0}(n)$ in $\cite{GirPal-2024}$ concerning a specific choice of reduced form, they did not prove that the definition of these representations is independent of the choice of a reduced form taken for a permutation $w\in \mathcal{S}_{n+1}$. To prove this, we now make use of the Coxeter structure of $\mathcal{S}_{n+1}$, in particular the Bruhat ordering $\leq $ of $\mathcal{S}_{n+1}$. For further details on the Bruhat order, we refer to \cite{BjoBre-2004}. In the proof of the next proposition, we will write $\pi^{(0)}_{n,w}$ for the representation $\pi^{(0)}_{w}$ of $A_{0}(n)$ to avoid notational confusion, and these notations will not be used elsewhere.
\begin{prop}
\label{factor}
   \textit{ For any two reduced expressions of $w_{1}$ and $w_{2}$ of a permutation $w$, $\pi^{(0)}_{w_{1}}$ and $\pi^{(0)}_{w_{2}}$ are unitarily equivalent. In particular, the collection of irreducible representations of $A_{0}(n)$ is indexed by the elements of $\mathbb{T}\times \mathcal{S}_{n+1}$.}
\end{prop}
\begin{proof}
As observed in Theorem 3.3.1 of $\cite{BjoBre-2004}$, given two reduced forms of an element of $S_{n+1}$ one can reach from one reduced form to another by performing a finite sequence of braid moves i.e. $s_{i}s_{i+1}s_{i}=s_{i+1}s_{i}s_{i+1}$ and $s_{i}s_{j}=s_{j}s_{i}$ for $|i-j|\geq 2, 1\leq i,j\leq n$. Keeping this in mind we proceed to prove the proposition first for the algebra $A_{0}(2)$. Define a $*$-automorphism $\phi$ of $A_{0}(2)$ by sending $z_{i,j}(0)\to z_{4-i,4-j}^{*}(0)$ for $1\leq i,j\leq 3$. Then one can directly verify $\pi^{(0)}_{2,s_{1}s_{2}s_{1}}= \pi^{(0)}_{2,s_{2}s_{1}s_{2}}\circ \phi $. Therefore, $\pi^{(0)}_{2,s_{2}s_{1}s_{2}}$ is irreducible and hence using Theorem $\ref{mu0}$, $\pi^{(0)}_{2,s_{2}s_{1}s_{2}}$ is equivalent to $\pi^{(0)}_{2,s_{1}s_{2}s_{1}}$. Let $U$ be an intertwining unitary between these two representations, that is $U^{*}\pi^{(0)}_{2,s_{2}s_{1}s_{2}}U=\pi^{(0)}_{2,s_{1}s_{2}s_{1}}$. The proof for $A_{0}(n)$ is done in two steps. For $1\leq i,j\leq n$ with $|i-j|\geq 2$, $\pi^{(0)}_{n,s_{i}s_{j}}\cong \pi^{(0)}_{n,s_{j}s_{i}}$ via the flip operator on $l^{2}(\mathbb{N})^{\otimes 2}$. For the other braid move i.e. $s_{i}s_{i+1}s_{i}=s_{i+1}s_{i}s_{i+1}$, observe that one can write $\pi^{(0)}_{n,s_{i}s_{i+1}s_{i}}=\pi^{(0)}_{2,s_{1}s_{2}s_{1}}\circ \phi_{1}$ and $\pi^{(0)}_{n,s_{i+1}s_{i}s_{i+1}}=\pi^{(0)}_{2,s_{2}s_{1}s_{2}}\circ \phi_{1}$ where $\phi_{1}: A_{0}(n)\to A_{0}(2)$ is a $*$-homomorphism that sends $z_{k,l}\to z_{k-i+1, l-i+1}$ if $k,l\in \{i,i+1,i+2\}$ and $\delta_{k,l}I$ on other generators. Therefore, the same unitary $U$ intertwines $\pi^{(0)}_{n,s_{i}s_{i+1}s_{i}}$ and $\pi^{(0)}_{n,s_{i+1}s_{i}s_{i+1}}$ and the rest of the proof follows easily. 
\end{proof}
\begin{rem}
 For $q\neq 0$, the above result is already known (see \cite{KorSoi-1998}), but observe that the above proof does not work for this case. So, from now on, we will freely use the notation $w$ for any reduced word irrespective of the choice of reduced form. 
\end{rem}
\section{Topology and Borel structure on $\hat{A}_{0}$}
Let $A$ be a separable $C^{*}$-algebra and $J$ be an ideal of $A$. $J$ is called primitive if $J$ is the kernel of some nonzero irreducible representation of $A$. Let $Prim(A)$ be the collection of all primitive ideals of $A$. We will work with the Jacobson topology on $Prim(A)$ whose closed sets can be characterized as follows. $S\subseteq Prim(A)$ is closed if and only if there is a subset $M$ of $A$ such that $S$ is \textit{exactly} the collection of all primitive ideals containing $M$. Let $\hat{A}$ be the collection of all equivalence classes of non-zero irreducible representations of $A$. The topology $\tau$ on $\hat{A}$ is then defined as pullback of the Jacobson topology via the canonical map from $\hat{A} \to Prim(A)$. This topology can be very pathological because it is not even Hausdorff in general (see Corollary \ref{t1}). However,  if we take $A$ to be a separable type I $C^{*}$-algebra then $(\hat{A}, \mathcal{B}_{\tau})$ is a standard Borel space (see 4.6.1, \cite{Dix-1977}) where $\mathcal{B}_{\tau}$ is the Borel sigma-algebra generated by pullback of the Jacobson topology on $\hat{A}$. Furthermore, the canonical map $\hat{A}\to Prim(A)$ becomes a homeomorphism (see 3.1.6, 9.1, \cite{Dix-1977}).\\
For any $q\in[0,1)$ we will denote the collection of equivalence classes of all one-dimensional representations and the collection of all equivalence classes of infinite-dimensional irreducible representations on $l^{2}(\mathbb{N})$ by $\hat{A}_{q,1}$ and $\hat{A}_{q,\infty}$ respectively. Let $w\neq 1$ be any reduced word. Denote $\hat{A}_{q,w}:=\{[V^{*}_{w}\pi^{(q)}_{t,w}( . )V_{w}]:\, t\in \mathbb{T}\}\subseteq \hat{A}_{q}$ where $V_{w}: l^{2}(\mathbb{N})\to l^{2}(\mathbb{N})^{\otimes l(w)}$ is a fixed unitary and $l(w)$ means the length of the reduced word. Also, since we are considering the class of irreducibles, the choice of unitary $V_{w}$ does not matter. Then using Theorems \ref{muq} and \ref{mu0} one can write $\hat{A}_{q}= \hat{A}_{q,1} \bigsqcup \hat{A}_{q,\infty} = \hat{A}_{q,1} \; \bigsqcup \; \left (\bigsqcup_{w\neq 1} \hat{A}_{q,w}\right )$. One of the goals of this section is to prove that the above decomposition is into Borel subsets. Hence, we have the following decomposition
\begin{align*}
    \int_{\hat{A}_{q}}^{\bigoplus} F \cong  \int_{\hat{A}_{q,1}}^{\bigoplus} F \; \bigoplus \; \left (\bigoplus_{w\neq 1} \int_{\hat{A}_{q,w}}^{\bigoplus}F\right ).
\end{align*}
We will use this decomposition with a suitable integrand $F$ later. Let $w\neq 1$ be a fixed reduced word in $\mathcal{S}_{n+1}$. We write $Irr_{\infty}(A_{q})$ for the collection of all non-zero infinite-dimensional irreducible representations of $A_{q}$ on $l^{2}(\mathbb{N})$ and the set $Irr^{}_{V_{w},w}(A_{q}):=\{V^{*}_{w}\pi^{(q)}_{t,w}( . )V_{w}:\, t\in \mathbb{T}\}\subseteq Irr_{\infty}(A_{q})$.
\begin{lema}
\label{polish}
\textit{
For $q\in [0,1)$ and a reduced word $w\neq 1$ we have 
\begin{enumerate}
\item $Irr_{\infty}(A_{q})$ with the topology of strong pointwise convergence is a polish space.
\item The natural map $\rho: Irr_{\infty}(A_{q})\to \hat{A}_{q,\infty}(\subset \hat{A}_{q})$ is open, continuous and onto. Hence, it is a Borel map from a standard space $Irr_{\infty}(A_{q})$ to a standard space $\hat{A}_{q}$.
\item $\rho(Irr^{}_{V_{w},w}(A_{q}))= \hat{A}_{q,w}$ and $\rho$ restricted to $Irr_{V_{w},w}(A_{q})$ is injective.
\end{enumerate}}
\end{lema}
\begin{proof}
    1 and 2 are obvious from results 3.7.4 and 3.5.8 in \cite{Dix-1977}. The equality $\rho(Irr^{}_{V_{w},w}(A_{q}))= \hat{A}_{q,w}$ follows immediately once we realise $l^{2}(\mathbb{N})$ and $l^{2}(\mathbb{N})^{\otimes l(w)}$ are isomorphic Hilbert spaces and injectivity of the restriction of $\rho$ is a consequence of Theorems \ref{muq} and \ref{mu0}.
\end{proof}
\begin{lema}
\label{lema2}
\textit{
Let $w\neq 1$ be a fixed reduced word. Then the map $f:Irr^{ }_{V_{w},w}(A_{q}) \to \mathbb{T}$ defined by $V^{*}_{w}\pi^{(q)}_{t,w}( . )V_{w}\to t$
is a homeomorphism.}
\end{lema}
\begin{proof}
Surjectivity is obvious, and injectivity again follows from Theorems \ref{muq} and \ref{mu0}. Now we prove that $f$ is continuous. Let $\{V^{*}_{w}\pi^{(q)}_{t_{\alpha},w}( . )V_{w}\}_{\alpha}$ be a net in $Irr^{ }_{V_{w},w}(A_{q})$ that converges to $V^{*}_{w}\pi^{(q)}_{t,w}( . )V_{w}$ under the topology of strong pointwise convergence. That means for the generator $z_{1,1}(q)\in A_{q}$, we have $\left|t^{1}_{\alpha}-t^{1}\right|\left\|   V^{*}_{w}\pi^{(q)}_{w}(z_{1,1}(q))V_{w}(\eta)\right\| \to 0$ where $\eta \in l^{2}(\mathbb{N})$. Given any $\eta \neq0$, as the latter quantity is always nonzero, we must have $t_{\alpha}^{1}\to t^{1}$ in $T$. Similarly, using other diagonal generators $z_{i,i}(q)$, one can show $t_{\alpha}^{i}\to t^{i}$ for all $i=2,3\cdots n$. To prove the continuity of the inverse, let us take a net $t_{\alpha}\to t$ in $\mathbb{T}$. Following the same steps, we can conclude that $\left\|  V^{*}_{w}\pi^{(q)}_{t_{\alpha},w}(a)V_{w}(\eta)- V^{*}_{w}\pi^{(q)}_{t,w}(a)V_{w}(\eta) \right\| \to 0$, for any $a=z_{k,l}(q)$ or $z_{k,l}^{*}(q)$ and $\eta \in l^{2}(\mathbb{N})$. Hence, the above convergence is also true for any $a\in A_{q}(n)$. 
\end{proof}
\vskip 0.2cm
\begin{prop}
\label{iso}
\textit{For any reduced word $w\neq 1$, $\hat{A}_{q,w}$ is a Borel subset of $\hat{A}_{q}$ and restriction of $\rho$ on $Irr_{V_{w},w}(A_{q})$ is a Borel isomorphism between $Irr_{V_{w},w}(A_{q})$ and $\hat{A}_{q,w}$. Furthermore, $\hat{A}_{q,w}$ is Borel isomorphic to $\mathbb{T}$ for all $q\in[0,1)$.}
\end{prop}
\begin{proof}
From Lemma $\ref{lema2}$ it follows that $Irr^{ }_{V_{w},w}(A_{q})$ is a compact subset of a Polish space $Irr_{\infty}(A_{q})$. So, $Irr^{ }_{V_{w},w}(A_{q})$ is itself closed and hence a Borel subset of $Irr_{\infty}(A_{q})$. Therefore, the rest of the proof follows immediately from Corollary 15.2 of \cite{Kec-1994} or Appendix B21 of \cite{Dix-1977}. 
\end{proof}
\noindent Combining the previous results, now we can safely write that for $q\in (0,1)$, the map $T_{q}^{-1}: \hat{A}_{q,\infty}\to \hat{A}_{0,\infty}$ defined by $[V^{*}_{w}\pi^{(q)}_{t,w}( . )V_{w}]\to [V^{*}_{w}\pi^{(0)}_{t,w}( . )V_{w}]$ is a Borel isomorphism. It further extends to a Borel isomorphism between $\hat{A}_{q}$ and $\hat{A}_{0}$ naturally. We record this observation in the next theorem. 
\begin{thm}
\label{iso1}
\textit{For any $q\in (0,1)$ the map $T_{q}: \hat{A}_{0} \to \hat{A}_{q}$ is a Borel isomorphism}.
\end{thm} 
\begin{rem}
It is important to note that $\hat{A}_{0}$ and $\hat{A}_{q}$ are always Borel isomorphic, since both are standard Borel spaces with the same cardinality. But that does not imply that the map $T_{q}^{-1}:\hat{A}_{q}\to \hat{A}_{0}$ is a Borel isomorphism. We need this specific map $T_{q}$ to be our required isomorphism for the main result. Although $\hat{A}_{q}$ is Borel isomorphic to $\bigsqcup_{(n+1)!}\mathbb{T}$ for any $q\in [0,1)$, they can never be homeomorphic as pullback of the Jacobson topology on $\hat{A}_{q}$ is not Hausdorff (see Corollary \ref{t1}).
\end{rem}
\begin{thm}
\label{Bfactor}
\textit{Let $u$ and $w$ be two permutations in $\mathcal{S}_{n+1}$ with $u\leq w$. Then $\pi^{(q)}_{t,u}$ factorizes through $\pi^{(q)}_{t,w}$ for any $t\in \mathbb{T}$ and $q\in [0,1)$.}
\end{thm}
\begin{proof}
We will denote the Toeplitz algebra by $\mathscr{T}$. Now using the subword property of the Bruhat ordering on $\mathcal{S}_{n+1}$ (Theorem 2.2.2, $\cite{BjoBre-2004}$) one can obtain a reduced expression for $u$ by deleting some Coxeter generators in a reduced word expression for $w$. Therefore, there is a $*$-homomorphism $ev:\mathscr{T}^{l(w)} \to \mathscr{T}^{l(u)}$ that sends $S\to 1$ for the deleted components in the reduced expression for $w$ and all other remaining components $S\to S$. Now, together with Proposition $\ref{factor}$, it is easy to see that $ev \circ \pi^{(q)}_{t,w}=\pi^{(q)}_{t,u}$. 
\end{proof}
\noindent It should also be noted that $\pi^{(q)}_{t,u}$ factorizes through $\pi^{(q)}_{t,w}$ if and only if $Ker\, \pi^{(q)}_{t,w}\subseteq Ker\,\pi^{(q)}_{t,u}$ for any $t\in \mathbb{T}$. For $q\in (0,1)$, the converse of Theorem $\ref{Bfactor}$ also holds for $A_{q}(n)$ (Theorem 4.1, \cite{NesTus-2012}) but it is not clear whether the same is true for $A_{0}(n)$.
\begin{cor}
\label{t1}
  \textit{For any $q\in[0,1)$, the topology of $\hat{A}_{q}$ is not Hausdorff}.
\end{cor}
\begin{proof}
Choose a $t\in \mathbb{T}$ and a reduced word $w_{2} \in \mathcal{S}_{n+1}$. If $w_{1}$ is a Bruhat subword of $w_{2}$, then there is a $C^*$-homomorphism $\phi:\mathscr{T}^{l(w_{2})}\to \mathscr{T}^{l(w_{1})}$ such that
$\phi\circ \pi_{t,w_{2}}^{(q)}=\pi_{t,w_{1}}^{(q)}$. That means $Ker\,\pi_{t,\omega_2}^{(q)}\subseteq Ker\,\pi_{t,\omega_1}^{(q)}$. Therefore, from the description of closed sets of the Jacobson topology on $Prim(A_{q})$ it follows that the set $\{Ker\,\pi_{t,\omega_2}^{(q)}\}$ is not closed. Hence, its pullback $\displaystyle \{[V_{w_{2}}^{*}\pi_{t,\omega_2}^{(q)}( . )V_{w_{2}}]\}\subseteq \hat{A}_{q}$, which is a singelton set, is again not closed.
\end{proof}
\noindent We will use the notation $w_{L}= s_{1}(s_{2}s_{1})(s_{3}s_{2}s_{1})\cdots (s_{n}s_{n-1}\cdots s_{2}s_{1})$ for the longest word in $\mathcal{S}_{n+1}$. Any reduced word $w$ is a Bruhat subword of $w_{L}$. 
\begin{cor}
\label{density}
\textit{For all $q\in [0,1)$, $\hat{A}_{q,w_{L}}$ is dense in $\hat{A}_{q}$.}
\end{cor}
\begin{proof}
Since any reduced word is a Bruhat subword of the longest word $w_{L}$, $\bigcap_{t} Ker\,\pi^{(q)}_{t,w_{L}}=\{0\}$. This follows from the fact that no non-zero element can lie in the kernel of every irreducible representation of a $C^{*}$-algebra. Therefore, any element $[\pi] \in \hat{A}_{q}$ is in the closure of $\hat{A}_{q,w_{L}}$ and the density of $\hat{A}_{q,w_{L}}$ follows. 
\end{proof}
\section{Main results}
Our goal of this section is to extend Theorem \ref{irr} for any non-degenerate (respectively faithful) representations of $A_{0}(n)$ and describe $A_{0}(n)$ as the $C^{*}$-algebra generated by the limit operators of faithful representations of $A_{q}(n)$. We need a decomposition theorem for representations involving disintegration to prove these results, which we will discuss first. We refer to chapter 8 of \cite{Dix-1977} for more details. Let $Z$ be a Borel space and $\mathcal{H}_{n}$ be the standard $n$-dimensional Hilbert space. By $A$ we denote a separable $C^{*}$-algebra of type I. $\hat{A}$ stands for its spectrum and $\hat{A}_{n}$ stands for the collection of equivalence classes of irreducible representations of $A$ on $\mathcal{H}_{n}$.
\begin{defn}
A constant field $\Gamma_{n}$ of Hilbert spaces $\tau \to \mathcal{H}(\tau)$ defined by $\mathcal{H}_{n}$ is a field with the following properties,
\begin{itemize}
    \item $\mathcal{H}(\tau)= \mathcal{H}_{n}$ for all $\tau \in Z$.
    \item $\Gamma_{n}=\{ \tau\to x(\tau): \tau \to \langle x(\tau)|a \rangle$ is Borel for all $a\in \mathcal{H}_{n}\}$.
\end{itemize}
\end{defn}
Then for every positive Borel measure $\mu$ on $Z$ one has $\Gamma_{n}\subseteq \Gamma_{n}^{\mu}$ where
\begin{center}
    $\Gamma_{n}^{\mu} = \{\tau\to x(\tau): \tau\to \langle x(\tau)|a\rangle$ is $\mu$-measurable for all $a\in \mathcal{H}_{n} \}$. 
\end{center}
Let $n$ be a cardinal and $n\leq \aleph_{0}$. There is a unique Borel field of Hilbert spaces $\Gamma_{c}: \tau \to \mathcal{H}(\tau)$ on $\hat{A}$ such that $\Gamma$ reduces on $\hat{A}_{n}$ to the constant field defined by $\mathcal{H}_{n}$. For any positive Borel measure $\mu$ on $\hat{A}$, $\Gamma_{c}$ is endowed with a $\mu$-measurable field structure that is $\Gamma_{c}\subseteq \Gamma_{c}^{\mu}$ where
\begin{center}
    $\Gamma_{c}^{\mu}=\{ \tau\to x(\tau): (\tau\to x(\tau))|_{\hat{A}_{n}}\in \Gamma_{n}^{\mu}\}$.
\end{center}
We will frequently talk about representations of $A$ on the Hilbert space $\displaystyle \int_{\hat{A}}^{\bigoplus}\mathcal{H}(\tau)d\mu$ defined by the field $((\mathcal{H}(\tau))_{\tau\in \hat{A}}, \Gamma_{c}^{\mu})$.
\begin{thm}[8.6.2, \cite{Dix-1977}]
\label{br1}
\textit{1. There is a field $\tau \to \psi(\tau)$ of irreducible representations of $A$ (defined on $\hat{A}$) in the Hilbert spaces $\mathcal{H}(\tau)$ such that $\psi(\tau)\in \tau$ and this field is $\mu$-measurable for any positive measure $\mu$ on $\hat{A}$.\\
2. If there is another field of irreducible representations $\tau \to \psi_{1}(\tau)$ with the same properties as above, then for every positive Borel measure $\mu$ on $\hat{A}$, we have
\begin{align*}
    \int_{\hat{A}}^{\bigoplus}\psi(\tau)d\mu \cong \int_{\hat{A}}^{\bigoplus}\psi_{1}(\tau)d\mu.
\end{align*}}
\end{thm}
\noindent This equivalence class of representations is denoted by $\displaystyle \int_{\hat{A}}^{\bigoplus}\tau d\mu$.
\begin{thm}[8.6.6, \cite{Dix-1977}]
\label{br2}
 \textit{   Let $A$ be a separable type I $C^{*}$-algebra and $\pi: A\to B(\mathcal{H})$ be a non-degenerate representation of $A$ on a separable Hilbert space $\mathcal{H}$.
    Then there are mutually singular positive Borel measures $\mu_{1},\mu_{2},...,\mu_{\infty}$ on $\hat{A}$ such that
    \begin{align*}
        \pi \cong \int_{\hat{A}}^{\bigoplus}\tau d\mu_{1}(\tau)\ \bigoplus \; 2 \int_{\hat{A}}^{\bigoplus}\tau d\mu_{2}(\tau) \bigoplus \cdots \bigoplus \; \aleph_{0}\int_{\hat{A}}^{\bigoplus}\tau d\mu_{\infty}(\tau).
    \end{align*}}
\end{thm}
\vskip 0.1cm
\noindent The field of irreducible representations that appeared in part 1 of Theorem $\ref{br1}$ is completely existential. Therefore we will directly construct a field of irreducible representations $\tau_{q} \to \psi_{q}(\tau_{q})$ of $A_{q}$ which satisfies all the properties of part 1 of Theorem \ref{br1} and hence by part 2 of Theorem $\ref{br1}$ we can conclude that $\displaystyle \int_{\hat{A}_{q}}^{\bigoplus}\tau_{q} d\mu(\tau_{q}) = \left[ \int_{\hat{A}_{q}}^{\bigoplus}\psi_{q}(\tau_{q})d\mu(\tau_{q})\right].$
\vskip 0.2cm
\noindent Let $\mathcal{H}_{1},\mathcal{H}_{2},\cdots \mathcal{H}_{\infty}$ be standard Hilbert spaces of dimension 1,2,...$\aleph_{0}$. For each cardinal $n\leq \aleph_{0}$, the collection of all representations of $A$ on $\mathcal{H}_{n}$ is denoted by $Rep_{n}(A)$. The Borel structure on $Rep_{n}(A)$ is defined by the topology of strong pointwise convergence. One can equip $Rep(A)= \bigcup_{n\leq \aleph_{0}} Rep_{n}(A)$ with the sum Borel structure and $Irr(A)$ is then a Borel subset of $Rep(A)$ because for each $n$, $Irr_{n}(A)$ is a Borel subset of $Rep_{n}(A)$ (see 3.8, $\cite{Dix-1977}$).
\begin{prop}[8.1.8, \cite{Dix-1977}]
\label{decom}
   \textit{Let $(Z,\mathcal{B},\mu)$ be a Borel space and $\tau \to \mathcal{H}(\tau)$ be a $\mu$-measurable field of Hilbert spaces defined on $Z$. Suppose that $Z$ is the union of mutually disjoint Borel sets $Z_{1},Z_{2},...Z_{\infty}$ and that $\tau \to \mathcal{H}(\tau)$ reduces on $Z_{n}$ to the constant field $\tau \to \mathcal{H}_{n}$. Let $\tau\to \pi(\tau)$ be a field of representations of $A$ on the Hilbert spaces $\mathcal{H}(\tau)$. Then $\tau \to \pi(\tau)$ is $\mu$-measurable if and only if $\tau \to \pi(\tau)$ is equal $\mu$-almost everywhere to a Borel mapping of $Z$ into $Rep(A)$.}
\end{prop}
\noindent For each $q\in [0,1)$, $A_{q}$ is a separable $C^{*}$-algebra of type I (see \cite{Bra-1990}, \cite{GirPal-2024}). Define a field of irreducible representations $\tau_{q}\to \psi_{q}(\tau_{q})$ on $\hat{A}_{q}$ as follows,
\begin{align*}
    \psi_{q}(\tau_{q})= \left\{ 
    \begin{matrix}
        V^{*}_{w}\pi^{(q)}_{t,w}( . )V_{w} & \tau_{q}=[V^{*}_{w}\pi^{(q)}_{t,w}( . )V_{w}]\\
        \chi_{t}^{(q)} & \tau_{q}= [\chi_{t}^{(q)}]
        \end{matrix}\right.
\end{align*}
\begin{cor}
\label{exist}
 \textit{Let $q\in [0,1)$. Then $\tau_{q}\to \psi_{q}(\tau_{q})$ is a $\mu$-measurable field of representations for any positive Borel measure $\mu$ on $\hat{A}_{q}$.}
\end{cor}
\begin{proof}
    Restrict $\tau_{q}\to \psi_{q}(\tau_{q})$ to $\hat{A}_{q,w}$. Then by Proposition \ref{iso}
    this restricted map is a Borel map from $\hat{A}_{q,w}$ into $Irr^{ }_{V_{w},w}(A_{q})(\subseteq Irr(A_{q})\subseteq Rep(A_{q}))$. Since $\hat{A}_{q,1}$ and $\hat{A}_{q,w}$ for $w\neq 1$ are all disjoint Borel subsets of $\hat{A}_{q}$, it is easy to show that $\tau_{q}\to \psi_{q}(\tau_{q})$ is a Borel mapping from $\hat{A}_{q}$ into $Rep(A_{q})$. Hence, the result follows if we take $Z= \hat{A}_{q}, Z_{1}=\hat{A}_{q,1},Z_{\infty}=\hat{A}_{q,\infty}$ and the empty set to be the null set appearing in Proposition \ref{decom}.
\end{proof}
\begin{rem}
    Therefore, it follows Theorem \ref{br1} that $\displaystyle \int_{\hat{A}_{q}}^{\bigoplus}\tau_{q} d\mu(\tau_{q}) = \left[ \int_{\hat{A}_{q}}^{\bigoplus}\psi_{q}(\tau_{q})d\mu(\tau_{q})\right]$, for any positive Borel measure $\mu$ on $\hat{A}_{q}$ and for all $q\in [0,1)$. The key observation here is that the unitary $V_{w}$'s used to define this field of irreducibles is independent of $q\in [0,1)$ and plays a crucial role in the proof of the next theorem. 
\end{rem}
\noindent Now we will extend Theorem \ref{irr} for any non-degenerate representation of the crystallized algebra. 
\begin{thm}
\label{main}
 \textit{
Let $\pi^{(0)}$ be a non-degenerate representation of $A_{0}$
on a seperable Hilbert space $\mathcal{H}$. Then there exist non-degenerate
representations $\pi^{(q)}$ of $A_{q}$ 
on the same Hilbert space $\mathcal{H}$ such that
}
\begin{IEEEeqnarray*}{rCl}
\pi^{(0)}(z_{k,j}^{}(0))=
 \lim_{q\to 0+}
 \pi^{(q)}\left((-q)^{\min\{k-j,0\}}z_{k,j}^{}(q)\right),\qquad
  k,j\in\{1,2,\ldots,n+1\}.
\end{IEEEeqnarray*}
\end{thm}
\begin{proof}
     Let $\mu_{1},\mu_{2},...\mu_{\infty}$ be mutually singular positive measures on $\hat{A}_{0}$ corresponding to the following decomposition of $\pi^{(0)}$ (Theorem \ref{br2})
\begin{align*}
   \int_{\hat{A}_{0}}^{\bigoplus}\psi_{0}(\tau_{0}) d\mu_{1} \bigoplus \; 2 \int_{\hat{A}_{0}}^{\bigoplus}\psi_{0}(\tau_{0}) d\mu_{2} \bigoplus \cdots \bigoplus \; \aleph_{0}\int_{\hat{A}_{0}}^{\bigoplus}\psi_{0}(\tau_{0}) d\mu_{\infty}.
\end{align*}
For $q\in (0,1)$, define a representation $\pi^{(q)}$ of $A_{q}$ as follows,
\begin{align*}
    \pi^{(q)}:= \int_{\hat{A}_{0}}^{\bigoplus}\psi_{q}(T_{q}(\tau_{0})) d\mu_{1} \bigoplus \; 2 \int_{\hat{A}_{0}}^{\bigoplus}\psi_{q}(T_{q}(\tau_{0})) d\mu_{2} \bigoplus \cdots \bigoplus \; \aleph_{0}\int_{\hat{A}_{0}}^{\bigoplus}\psi_{q}(T_{q}(\tau_{0})) d\mu_{\infty}
\end{align*}
where $T_{q}: \hat{A}_{0}\to \hat{A}_{q}$ is the Borel isomorphism of Theorem \ref{iso1} discussed earlier. $\pi^{(q)}$ is well defined that follows from Theorem \ref{iso1} and Corollary \ref{exist}. 
Let $z_{k,j}(0)$ be a generator of $A_{0}$. For any $\varepsilon >0$, there is a natural number $m\in \mathbb{N}$ such that whenever $q< \frac{1}{m}$ we have,
\begin{center}
    $\left\|\pi^{(q)}_{t,w}(c_{k,j}(q)z_{k,j}(q))-\pi^{(0)}_{t,w}(z_{k,j}(0))\right\|< \varepsilon$
\end{center}
for all $t\in \mathbb{T}$ and for all reduced words $w$ and $c_{k,j}(q)=(-q)^{\min\{k-j, 0\}}$. Note that we have chosen the natural number $m$ independent of $t\in \mathbb{T}$ and reduced word $w$. This follows from the way $\pi^{(q)}_{t,w}$ is defined for all  $q\in [0,1)$. For any positive measure $\mu_{i}$ as above, we can write
\begin{align*}
    \int_{\hat{A}_{0}}^{\bigoplus}\psi_{0}(\tau_{0}) d\mu_{i}\cong \int_{\hat{A}_{0,1}}^{\bigoplus}\psi_{0}(\tau_{0}) d\mu_{i}\; \bigoplus \; \left(\bigoplus_{w\neq 1} \int_{\hat{A}_{0,w}}^{\bigoplus}\psi_{0}(\tau_{0}) d\mu_{i}\right).
\end{align*}
Similarly, we can break $\int_{\hat{A}_{0}}^{\bigoplus}\psi_{q}(T_{q}(\tau_{0})) d\mu_{i}$ using Proposition \ref{iso}. It is easy to see that 
\begin{align*}
    \left\|\int_{\hat{A}_{0,1}}^{\bigoplus}\psi_{q}(T_{q}(\tau_{0}))(c_{k,j}(q)z_{k,j}(q)) d\mu_{i}-\int_{\hat{A}_{0,1}}^{\bigoplus}\psi_{0}(\tau_{0})(z_{k.j}(0)) d\mu_{i}\right\|=0.
\end{align*}
Let $\tau_{0}=[V^{*}_{w}\pi^{(0)}_{t,w}( . ) V_{w}]\to x(\tau_{0})$ be a unit vector field i.e., $\displaystyle \int_{\hat{A}_{0,w}}\left\|x(\tau_{0})\right\|^{2}d\mu_{i}=1$. Then 
\begin{IEEEeqnarray*}{rCl}
    \IEEEeqnarraymulticol{3}{l}{\qquad\qquad 
    \left\|\left[ \int_{\hat{A}_{0,w}}^{\bigoplus}\psi_{q}(T_{q}(\tau_{0}))(c_{k,j}(q)z_{k,j}(q)) d\mu_{i}-\int_{\hat{A}_{0,w}}^{\bigoplus}\psi_{0}(\tau_{0})(z_{k,j}(0))d\mu_{i}\right](x(\tau_{0}))\right\|^{2} \qquad \qquad\qquad\qquad}\\
\qquad\qquad \qquad \qquad &=& \int_{\hat{A}_{0,w}}\left\|V^{*}_{w}\left[\pi^{(q)}_{t,w}(c_{k,j}(q)z_{k,j}(q))-\pi^{(0)}_{t,w}(z_{k,j}(0))\right]V_{w}x(\tau_{0})\right\|^{2}d\mu_{i} \\
& < & \varepsilon^{2}. \int_{\hat{A}_{0,w}}\left\|x(\tau_{0})\right\|^{2}d\mu_{i}\\ 
&< &\varepsilon^{2}.
\end{IEEEeqnarray*}
That means for any reduced word $w\neq 1$,
\begin{align*}
\left\|\int_{\hat{A}_{0,w}}^{\bigoplus}\psi_{q}(T_{q}(\tau_{0}))(c_{k,j}(q)z_{k,j}(q)) d\mu_{i}-\int_{\hat{A}_{0,w}}^{\bigoplus}\psi_{0}(\tau_{0})(z_{k,j}(0)) d\mu_{i}\right\|< \varepsilon
\end{align*}
whenever $q< \frac{1}{m}$.
Hence, we can write
\begin{align*}
    \left\|\int_{\hat{A}_{0}}^{\bigoplus}\psi_{q}(T_{q}(\tau_{0}))(c_{k,j}(q)z_{k,j}(q)) d\mu_{i}-\int_{\hat{A}_{0}}^{\bigoplus}\psi_{0}(\tau_{0})(z_{k,j}(0)) d\mu_{i}\right\|< \varepsilon.
\end{align*}
Since the above inequality is true for all positive measures $\mu_{i}$'s appearing in the decomposition of $\pi_{0}$, therefore for any $q< \frac{1}{m}$ we have
\begin{center}
    $\left\| \pi^{(q)}(c_{k,j}(q)z_{k,j}(q))-\pi^{(0)}(z_{k,j}(0))\right\|\leq \varepsilon$.
\end{center}
Non-degeneracy of $\pi^{(q)}$ follows from the fact that irreducible representations are non-degenerate and a direct integral of non-degenerate representations is again non-degenerate (8.1.4, \cite{Dix-1977}). This concludes the proof of the theorem.
\end{proof}
\noindent Our next goal is to \textit{start with a faithful representation $\pi^{(0)}$ and prove that we can choose $\pi^{(q)}$'s to be faithful}. For this, we need some further results. The very first step is to observe that $\hat{A}_{0,w_{L}}$ is an open dense subset of $\hat{A}_{0}$. For any $1\leq k\leq n$ we will write $w_{k}$ for the reduced word $(s_{1})(s_{2}s_{1})\cdots(s_{k-1}s_{k-2}\cdots s_{1})(s_{k}s_{k-1}\cdots s_{2})(s_{k+1}s_{k}\cdots s_{2}s_{1})\cdots (s_{n}s_{n-1}\cdots s_{2}s_{1})$. That is the generator $s_{1}$ is only missing at the $k$-th bracket.
\begin{prop}
\label{key}
\textit{
\begin{enumerate}
\item$(z^{ }_{n+1,1})(z^{ }_{n,1}z^{ }_{n+1,2})(z^{ }_{n-1,1}z^{ }_{n,2}z^{ }_{n+1,3})\cdots (z^{ }_{2,1}z^{ }_{3,2}z^{ }_{4,3}\cdots z^{ }_{n+1,n})\in Ker\,\pi^{(0)}_{t,w}$ for all $t\in \mathbb{T}$ and $w\neq w_{L}$. We will denote this element by $a$.
\item $\pi^{(0)}_{w_{L}}(a)$ is a nonzero operator that is an elementary tensor product of $P_{0}$.  Hence $\hat{A}_{0,w_{L}}$ is an open dense subset of $\hat{A}_{0}$.
\item $\pi^{(0)}_{w_{L}}$ applied to any bracket element of $a$ is again a positive operator consisting of an elementary tensor product of $P_{0}$ and $I$. 
\end{enumerate}}
\end{prop}
\begin{proof}
1. From the relations of $A_{0}(n)$ (see $\cite{GirPal-2023}$) it follows that $z^{*}_{1,n+1}= (z^{ }_{2,1}z^{ }_{3,2}z^{ }_{4,3}\cdots z^{ }_{n+1,n})\in Ker\,\pi^{(0)}_{t,w_{1}}$ with $w_{1}=(s_{2}s_{1})(s_{3}s_{2}s_{1})\cdots (s_{n}s_{n-1}\cdots s_{1})$. Similarly, by direct computation we can prove that $z^{ }_{3,1}z^{ }_{4,2}\cdots z^{ }_{n+1,n-1}\in Ker\, \pi^{(0)}_{t,w_{2}}$ with $w_{2}= (s_{1})(s_{2})(s_{3}s_{2}s_{1})\cdots(s_{n}s_{n-1}\cdots s_{1})$ so on and finally $z^{ }_{n+1,1}\in Ker\, \pi^{(0)}_{t,w_{n}}$ with $ w_{n}=(s_{1})(s_{2}s_{1})(s_{3}s_{2}s_{1})\cdots (s_{n}s_{n-1}\cdots s_{2})$. Since $\{w_{1},w_{2},\cdots w_{n}\}$ exhausts the list of all maximal Bruhat subwords of $w_{L}$, the rest of the proof follows from Theorem $\ref{Bfactor}$. The proof of both part 2 and part 3 follows from direct calculations. To prove $\hat{A}_{0.w_{L}}$ is open observe that its complement in $\hat{A}_{0}$ is closed which follows easily from the description of closed sets under the Jacobson topology.  
\end{proof} 
\noindent Let $\pi^{(0)}$ be a non-degenerate representation of $A_{0}$ on a separable Hilbert space. For the sake of simplicity, let us further assume $\pi^{(0)}$ is just a single direct integral of irreducibles over $\hat{A}_{0}$ with respect to some non-negative Borel measure $\mu$. The general case follows similarly with slight modifications. The next theorem helps us to filter out the minimal faithful part of a faithful representation of $A_{0}$.
\begin{thm}
\textit{ $\pi^{(0)}$ is faithful if and only if $Q_{L}:= \displaystyle \int_{\hat{A}_{0,w_{L}}}^{\bigoplus} V^{*}_{w_{L}}\pi^{(0)}_{t,w_{L}}( . )V_{w_{L}} d\mu$ is faithful.}
\label{faith}
\end{thm} 
\begin{proof}
 Suppose that $\pi^{(0)}$ is faithful. Then using part 1 of Proposition \ref{key} it follows that $\mu(\hat{A}_{0,w_{L}})>0$. Consider an element $z$ such that $Q_{L}(z)=0$. For showing $z=0$, it suffices to prove that for any $s\in \mathbb{T}$, $\pi^{(0)}_{s,w_{L}}(z)(e_{m_{1}}\otimes e_{m_{2}}\otimes \cdots \otimes e_{m_{n(n+1)/2}})=0$. Since $A_{0}$ is of type I therefore we have an element $y\in A_{0}$ such that $\pi^{(0)}_{s,w_{L}}(y)= (S^{*})^{m_{1}}P_{0}\otimes \cdots \otimes (S^{*})^{m_{n(n+1)/2}}P_{0}$. Now, observe that $zya \in Ker\, \pi^{(0)}$. Since $\pi^{(0)}$ is faithful, one has $zya=0$. Now use part 2 of the Proposition \ref{key} to conclude $\pi^{(0)}_{s,w_{L}}(z)(e_{m_{1}}\otimes e_{m_{2}}\otimes \cdots \otimes e_{m_{n(n+1)/2}})=0$. Converse is trivial. 
\end{proof}
\noindent  For a separable $C^{*}$-algebra $A$, it follows from Proposition 3.3.4 of $\cite{Dix-1977}$ that the topology of the spectrum $\hat{A}$ has a countable base. Therefore, for any positive Borel measure $\mu$ defined on the Borel $\sigma$-algebra generated by the topology of $\hat{A}$, the support of $\mu$ (the smallest closed set whose complement is negligible) is well defined. Alternately, support can be described by the set of all $\rho \in \hat{A}$ with the property that any open neighborhood of $\rho$ has a positive measure. The next proposition captures the faithfulness of a direct integral of irreducible representations in terms of the support of the underlying positive Borel measure.
\begin{prop}
\label{support}
\textit{ Let $A$ be any separable type I $C^{*}$-algebra and $\hat{A}$ be the spectrum of $A$. Then any non-degenerate representation $\pi:= \displaystyle \int_{\hat{A}}^{\bigoplus}\psi(\tau) d\mu$ (see Theorem \ref{br1}) is faithful if and only if supp$(\mu)=\hat{A}$. }
\end{prop}
\begin{proof}
    Assume $\pi$ is faithful. Suppose there is an irreducible representation $\rho \not\in$ supp$(\mu)$. As supp$(\mu)$ is a closed subset of $\hat{A}$, using the definition of Jacobson topology one can find a non-zero element $x\in A$ such that $\sigma(x)=0$ for all $\sigma \in$ supp$(\mu)$ and $\rho(x)\neq 0$. But this implies $\pi(x)=0$ which contradicts the faithfulness of $\pi$.\\
    Conversely, let us assume $\mu$ has full support but $\pi$ is not faithful. Then there is a non-zero element $x\in A$ such that $\pi(x)=0$. Consider the following subset $U:=\{ \tau \in \hat{A}:\space \psi(\tau)(x)\neq 0\}$ of $\hat{A}$. $U$ is non-empty because an irreducible representation sends $x$ to a non-zero operator as $x$ is itself non-zero. Also observe that $U$ is open in $\hat{A}$ and $\mu(U)=0$. But this contradicts our assumption $\mu$ has full support. 
\end{proof}
\begin{cor}
\textit{Any non-degenerate representation $\pi$ with decomposition as in Theorem $\ref{br2}$ is faithful if and only if supp$(\sum_{i}\mu_{i})=\hat{A}$. }
\end{cor}
\begin{proof}
Note that for each $i$, supp$(\mu_{i})\subseteq$ supp$(\sum_{j}\mu_{j})$ and the rest of the proof follows similarly as above.  
\end{proof}
\noindent 
For $q\in (0,1)$, the topology on the spectrum of $C(K_{q})$ was studied in great detail by Neshveyev \& Tuset for any simply connected semisimple compact Lie group $K$ with complexified Lie algebra $g$ in Section 4 of $\cite{NesTus-2012}$. They described this topology using the Bruhat order on the Weyl group of $g$ and certain subsets of the maximal torus of $K$. Currently, it is not clear to us whether $\hat{A}_{0}$ and $\hat{A}_{q}$ are homeomorphic, but for our subsequent results on the faithfulness of the representations of $A_{q}$ for different values of $q$, it suffices to show that the restriction of the natural Borel isomorphism $T_{q}$ to the dense open subset $\hat{A}_{0,w_{L}}$ is indeed a homeomorphism between $\hat{A}_{0,w_{L}}$ and $\hat{A}_{q,w_{L}}$. The following lemma is an easy consequence of Theorem 4.1 in $\cite{NesTus-2012}$.
\begin{lema}
\label{top}
\textit{Let $q\in (0,1)$ and $\Omega \subseteq \mathbb{T}$ and $w_{L}$ be the longest word. For any $t\in\mathbb{T}$, the kernal of the representation $\pi^{(q)}_{t,w_{L}}$ contains the intersection of the kernals of the representations $\pi^{(q)}_{s,w_{L}}$ of $A_{q}(n)$ for each $s\in \Omega$ if and only if $t\in \overline{\Omega}$. Therefore, a subset $M_{\Omega}:=\{[V_{w_{L}}^{*}\pi^{(q)}_{s,w_{L}}(.)V_{w_{L}}]: s\in \Omega \}$ of $\hat{A}_{q,w_{L}}$ is closed in $\hat{A}_{q,w_{L}}$ if and only if $\Omega$ is closed.}
\end{lema}
\begin{rem}
    Note that in \cite{NesTus-2012}, the authors have used the opposite indexing for the family of irreducible representations of $A_{q}(n)$, that is, each irreducible representation is denoted by $\pi^{(q)}_{w,t}:= \pi^{(q)}_{w}*\chi_{t}$ for a reduced word $w$ and $t=(t_{1}.t_{2},\cdots,t_{n})\in \mathbb{T}$. Therefore, we need to keep in mind the equivalence $\pi^{(q)}_{w_{L}}*\chi_{t}\cong \chi_{s}*\pi^{(q)}_{w_{L}}$ where $s=(\overline{t_{n}},\overline{t_{n-1}},\cdots, \overline{t_{1}})\in \mathbb{T}$ to prove the above result. 
\end{rem}
\noindent We now proceed to prove that the restriction of $T_{q}$ on $\hat{A}_{0,w_{L}}$ is a homeomorphism. It is already known that the natural map from $\mathbb{T} \to \hat{A}_{q,w_{L}}$ is a continuous bijection for $q\in [0,1)$ (see Lemma \ref{polish}). Therefore, our required homeomorphism can be achieved if we can prove $\hat{A}_{0,w_{L}}$ is Hausdorff because homeomorphism will then automatically follow as for $q\neq 0$, it is easy to see from Lemma $\ref{top}$ that $\hat{A}_{q,w_{L}}$ is homeomorphic with $\mathbb{T}$. The map $\rho: Irr_{\infty}(A_{0})\to \hat{A}_{0,\infty}$ is open continuous onto map (see Lemma \ref{polish}). Observe that the following set $Irr_{w_{L},\infty}(A_{0}):=\{ W^{*}\pi^{(0)}_{t,w_{L}}( . )W:\, t\in \mathbb{T} ,\,\text{W: $l^{2}(\mathbb{N})\to l^{2}(\mathbb{N})^{\otimes l(w_{L})}$ is unitary } \} = \rho^{-1}(\hat{A}_{0,w_{L}})$ is a saturated open set. Therefore, $\rho|_{res}: Irr_{w_{L},\infty}(A_{0})\to \hat{A}_{0,w_{L}}$ is again an open quotient map defined on a Hausdorff space.
\begin{lema}
\textit{Let $R$ be an equivalence relation on a Hausdorff topological space $X$ and $p: X \to X/R$ is an open quotient map. If the relation $R$ is a closed relation that is $\{(x,y): xRy\}\subseteq X\times X$ is closed, $X/R$ is Hausdorff.}
 \end{lema} 
\noindent We will apply this result for $\rho|_{res}$. The equivalence relation will be closed in our case due to the following proposition. In the proof of the next proposition, we will use the notation $t= (t^{(1)},t^{(2)},\cdots,t^{(n)})$ for any $t\in \mathbb{T}$.
\begin{prop}
\textit{Let $U_{m}$ and $U$ be unitaries from $l^{2}(\mathbb{N})\to l^{2}(\mathbb{N})^{\otimes n(n+1)/2}$ such that for any $z\in A_{0}(n)$ and $x,y\in l^{2}(\mathbb{N})$ we have
$\lim_{m\to \infty} \langle x | U^{*}_{m}\pi^{(0)}_{t_{m},w_{L}}(z)U_{m}(y)\rangle = \langle x | U^{*}\pi^{(0)}_{t,w_{L}}(z)U(y)\rangle$. Then $\lim_{m\to \infty}t_{m}=t$.}
\end{prop}
\begin{proof}
Using the elements $z^{ }_{1,n+1}$ and $z^{ }_{n+1,1}$ one can show the first and last component of the sequence $t_{m}$ converges to the first and last component of $t$ respectively. For the convergence of other components of $t_{m}$, we need part 3 of the Proposition \ref{key} particularly the positivity of the operators. Choosing $x=y$ and bracket elements of $a$ suitably one can conclude that all other components of $t_{m}$ converge to components of $t$. For instance, let us focus on the convergence $t_{m}^{(n-1)}\to t^{(n-1)}$. Choose $x=y$ appropriately so that
\begin{center}
    $\overline{t^{(n-1)}}\langle x | U^{*}\pi^{(0)}_{w_{L}}(z^{ }_{n,1}z^{ }_{n+1,2})U(x)\rangle =\langle x | U^{*}\pi^{(0)}_{t,w_{L}}(z^{ }_{n,1}z^{ }_{n+1,2})U(x)\rangle \neq 0$.
\end{center} 
Therefore for all but finitely many $m$,
\begin{center}
    $\overline{t_{m}^{(n-1)}}\langle x | U^{*}_{m}\pi^{(0)}_{w_{L}}(z^{ }_{n,1}z^{ }_{n+1,2})U_{m}(x)\rangle = \langle x | U^{*}_{m}\pi^{(0)}_{t_{m},w_{L}}(z^{ }_{n,1}z^{ }_{n+1,2})U_{m}(x) \rangle \neq 0.$
\end{center}
Also by part 3 of Proposition $\ref{key}$, we have that the quantities $\langle x | U^{*}_{m}\pi^{(0)}_{w_{L}}(z^{ }_{n,1}z^{ }_{n+1,2})U_{m}(x)\rangle$ and $\langle x | U^{*}\pi^{(0)}_{w_{L}}(z^{ }_{n,1}z^{ }_{n+1,2})U(x) \rangle$ are both positive as $\pi^{(0)}_{w_{L}}(z^{ }_{n,1}z^{ }_{n+1,2})$ is a positive operator consisting of elementary tensor of $P_{0}$ and $I$. Therefore, we have
\begin{IEEEeqnarray*}{rCl}
\lim_{m\to \infty} \langle x | U^{*}_{m}\pi^{(0)}_{w_{L}}
    (z^{ }_{n,1}z^{ }_{n+1,2})U_{m}(x)\rangle 
      &=& \lim_{m\to \infty} \left| \langle x | 
            U^{*}_{m}\pi^{(0)}_{t_{m},w_{L}}(z^{ }_{n,1}
            z^{ }_{n+1,2})U_{m}(x)\rangle \right|\\
      &=& \left|\langle x | U^{*}\pi^{(0)}_{t,w_{L}}
            (z^{ }_{n,1}z^{ }_{n+1,2})U(x)\rangle\right| \\
& = &  \langle x | U^{*}\pi^{(0)}_{w_{L}}(z^{ }_{n,1}z^{ }_{n+1,2})U(x) \rangle.
\end{IEEEeqnarray*}
Now it is evident that $t_{m}^{(n-1)}\to t^{(n-1)}$. 
\end{proof} 
\noindent Therefore, the relative topology on $\hat{A}_{0,w_{L}}$ is Hausdorff despite $\hat{A}_{0}$ being non-Hausdorff.\\ 
We now proceed to prove that  $\displaystyle \int_{\hat{A}_{0,w_{L}}}^{\bigoplus} V^{*}_{w_{L}}\pi^{(q)}_{t,w_{L}}( . )V_{w_{L}} d\mu=\int_{\hat{A}_{0}}^{\bigoplus} \psi_{q}(T_{q}(\tau_{0})) d\nu = \pi^{(q)}$ is faithful if   $\displaystyle \int_{\hat{A}_{0,w_{L}}}^{\bigoplus} V^{*}_{w_{L}}\pi^{(0)}_{t,w_{L}}( . )V_{w_{L}} d\mu=\int_{\hat{A}_{0}}^{\bigoplus} \psi_{0}(\tau_{0}) d\nu = \pi^{(0)}$ is faithful where $\nu$ is a new Borel measure on $\hat{A}_{0}$ which is equal to $\mu$ on $\hat{A}_{0,w_{L}}$ and 0 otherwise. Using the standard change of variable formula one can write $\pi^{(q)} \cong \displaystyle \int_{\hat{A}_{q}}^{\bigoplus} \psi_{q}(\tau_{q})dT_{q}\nu$. So, for proving the faithfulness of $\pi^{(q)}$, it is enough to show that  $supp(T_{q}\nu)=\hat{A}_{q}$ (see Proposition \ref{support}). As support is a closed set, it further reduces to prove $\hat{A}_{q,w_{L}}\subseteq supp(T_{q}\nu)$ (see Corollary \ref{density}).  
\begin{prop}
$\label{q-0}$
\textit{$\pi^{(q)}$ is a faithful if and only if $\pi^{(0)}$ is faithful. }
\end{prop}
\begin{proof}
Assume that supp$(\nu)=\hat{A}_{0}$ but there is an element $x\in \hat{A}_{q,w_{L}}$ and $x$ is outside $supp(T_{q}\nu)$. That means we can find a non-empty open subset $V\subseteq \hat{A}_{q}$ , $x\in V$ and $T_{q}\nu(V)=0$. Hence $x\in V\cap \hat{A}_{q,w_{L}}$ is a non-empty open subset of $\hat{A}_{q,w_{L}}$ with $T_{q}\nu(V\cap \hat{A}_{q,w_{L}})=0$.  That implies $\nu(T_{q}^{-1}(V\cap \hat{A}_{q,w_{L}}))=0$. But $T_{q}^{-1}(V\cap \hat{A}_{q,w_{L}})$ is non-empty open in $\hat{A}_{0,w_{L}}$ (using the homeomorphism). By part 2 of the Proposition $\ref{key}$, $\hat{A}_{0,w_{L}}$ is open in $\hat{A}_{0}$. So, $T_{q}^{-1}(V\cap \hat{A}_{q,w_{L}})$ is a non-empty open subset of $\hat{A}_{0}$ with $\nu(T_{q}^{-1}(V\cap \hat{A}_{q,w_{L}}))=0$ which contradicts our assumption.\\
Converse follows along the same lines. 
\end{proof}
\noindent Now we are ready to formulate our concluding results. 
\begin{thm}
\label{main1}
\textit{Let $\pi^{(0)}$ be a non-degenerate faithful representation of $A_{0}$
on a seperable Hilbert space $\mathcal{H}$. Then there exist non-degenerate faithful
representations $\pi^{(q)}$ of $A_{q}$ 
on the same Hilbert space $\mathcal{H}$ such that
\begin{IEEEeqnarray*}{rCl}
\pi^{(0)}(z_{k,j}^{}(0))=
 \lim_{q\to 0+}
 \pi^{(q)}\left((-q)^{\min\{k-j,0\}}z_{k,j}^{}(q)\right),\qquad
  k,j\in\{1,2,\ldots,n+1\}.
\end{IEEEeqnarray*}}
\end{thm}
\begin{proof}
    Choose $\pi^{(q)}$ as given in the proof of the Theorem $\ref{main}$. Use Theorem $\ref{faith}$ to filter out the minimal faithful part of $\pi^{(0)}$. Faithfulness of $\pi^{(q)}$ now follows from the Propositions $\ref{support}$ and $\ref{q-0}$. 
\end{proof}
\noindent The next theorem establishes the converse that any non-degenerate faithful representation of $A_{q}(n)$ for $q\neq 0$ gives rise to a non-degenerate faithful representation of the crystallized algebra in a natural way.
\begin{thm}
\label{main2} 
\textit{For any $q_{0}\in (0,1)$, let $\pi^{(q_{0})}$ be a non-degenerate faithful representation of $A_{q_{0}}$
on a seperable Hilbert space $\mathcal{H}$. Then there exist non-degenerate faithful
representations $\pi^{(q)}$ of $A_{q}$ 
on the same Hilbert space $\mathcal{H}$ for all $q\in [0, q_{0})$ such that
\begin{IEEEeqnarray*}{rCl}
\pi^{(0)}(z_{k,j}^{}(0))=
 \lim_{q\to 0+}
 \pi^{(q)}\left((-q)^{\min\{k-j,0\}}z_{k,j}^{}(q)\right),\qquad
  k,j\in\{1,2,\ldots,n+1\}.
\end{IEEEeqnarray*}}
\end{thm}
\begin{proof}
To prove this, we need a $q\neq 0$ version of Proposition $\ref{faith}$, that is, $\pi^{(q)}\cong \displaystyle \int_{\hat{A}_{q}}^{\bigoplus} \psi_{q}(\tau_{q})d\mu$ is faithful if and only if $Q_{L}:= \displaystyle \int_{\hat{A}_{q,w_{L}}}^{\bigoplus}V^{*}_{w_{L}}\pi^{(q)}_{t,w_{L}}( . )V_{w_{L}} d\mu$ is faithful. It can similarly be proved by exploiting the fact that $\hat{A}_{q,w_{L}}$ is an open dense subset of $\hat{A}_{q}$ for $q\neq 0$ (Theorem 4.1, \cite{NesTus-2012}) and $A_{q}(n)$ is a type I $C^{*}$-algebra. Let $\pi^{(q)}$ be faithful and $z\in A_{q}$ be such that $Q_{L}(z)=0$. We have to prove that $z=0$.  It again suffices to show that for any $s\in \mathbb{T}$, $\pi^{(q)}_{s,w_{L}}(z)=0$. Now using the description of closed subsets of $\hat{A}_{q}$ we can find a non-zero element $a\in \bigcap_{[\pi]\in \hat{A}_{q}\setminus \hat{A}_{q,w_{L}}}Ker\,\pi$ but  $\pi^{(q)}_{s,w_{L}}(a)=K\neq 0$. Let $K_{1}$ be a any compact operator on $l^{2}(\mathbb{N})^{\otimes l(w_{L})}$. Since $A_{q}$ is of type I, there exists an element $z_{K_{1}}\in A_{q}$ such that $\pi^{(q)}_{s,w_{L}}(z_{K_{1}})=K_{1}$. Now observe that $zz_{K_{1}}a\in Ker\pi^{(q)}$. So we have  $zz_{K_{1}}a=0$. This again implies $\pi^{(q)}_{s.w_{L}}(z)K_{1}(K)=0$ where $K_{1}$ can be choosen to be any arbitrary compact operator. Hence it follows that $\pi^{(q)}_{s,w_{L}}(z)=0$.\\
The rest of the argument follows similarly by performing the steps mentioned in the proof of Theorem $\ref{main1}$ at $q\neq 0$ level, using the canonical homeomorphism between $\hat{A}_{0,w_{L}}$ and $\hat{A}_{q,w_{L}}$ and Proposition $\ref{support}$. 
\end{proof}
\begin{cor}
\label{cor}
\textit{Given any $q_{0}\in (0,1)$ and a non-degenerate faithful representation $\pi^{(q_{0})}$ of $C(SU_{q_{0}}(n+1))$ on a separable Hilbert space $\mathcal{H}$, there exist non-degenerate faithful representations $\pi^{(q)}$ of $C(SU_{q}(n+1))$ on the same Hilbert space $\mathcal{H}$ for all $q\in (0,q_{0})$ such that the following limits exist in norm
 \begin{IEEEeqnarray*}{rCl}
 \lim_{q\to 0+}
 \pi^{(q)}\left((-q)^{\min\{k-j,0\}}z_{k,j}^{}(q)\right)\in B(\mathcal{H}),\qquad
  k,j\in\{1,2,\ldots,n+1\}
\end{IEEEeqnarray*}
and the $C^{*}$-subalgebra of $B(\mathcal{H})$ generated by these limit operators is isomorphic to the crystallized algebra $C(SU_{0}(n+1))$.}
\end{cor}
\begin{rem}
    Let $\mu$ be the Haar measure on $\mathbb{T}$. Then, one can choose $\pi^{(q)}$ to be the faithful representation $\displaystyle \int_{\mathbb{T}}^{\bigoplus}\pi^{(q)}_{t,w_{L}}(.)d\mu$ (the Soibelman representation) of $C(SU_{q}(n+1))$, to conclude that the crystallized algebra $C(SU_{0}(n+1))$ is isomorphic to the $C^{*}$-algebra generated by the limit operators of the Soibelman representations of $C(SU_{q}(n+1))$'s as observed in Theorem 8.4 of $\cite{GirPal-2024}$.
\end{rem}
\noindent \textbf{Acknowledgement}- We are greatly indebted to Prof. Arup K. Pal for his valuable suggestions and constant support time and again. We would also like to thank the Indian Statistical Institute for supporting us both in the form of a Ph.D. fellowship. 

\vskip 0.2cm
\textbf{Saikat Das}\\
Theoretical Statistics and Mathematics Unit, Indian Statistical Institute, Delhi, India;\\
saikat20r@isid.ac.in, saikat1811@gmail.com
\vskip 0.6cm
\noindent \textbf{Ayan Dey}\\
Theoretical Statistics and Mathematics Unit, Indian Statistical Institute, Delhi, India;\\
ayan22r@isid.ac.in, studentayandey@gmail.com

\begin{thebibliography}{10}

                 \bibitem{BjoBre-2004} A. Bjorner and F. Brenti. {Combinatorics of Coxeter Groups}. Graduate texts in Mathematics 231, 2005, Springer.

                 \bibitem{Bra-1990} K. Bragiel. {The Twisted $SU(N)$ Group. On the $C^{*}$-Algebra $C(S_{\mu}U(N))$}.  Letters in Mathematical Physics 20: 251-257, 1990. 


                \bibitem{Dix-1977}Jacques Dixmier. {C$^{*}$-algebras}. Volume 15, 1977, North Holland publishing company. 
            

                \bibitem{GirPal-2023}Manabendra Giri and Arup Kumar Pal. {Quantized function algebras at q=0: type A$_{n}$ case.} Proc. Indian Acad. Sci. (Math. Sci.) 134, 30 (2024). 
   
                \bibitem{GirPal-2024}Manabendra Giri and Arup Kumar Pal. {Irreducible representations of the crystallization of the quantized function algebras $C(SU_{q}(n+1))$.} J. Noncommut. Geom. Doi 10.4171/JNCG/608.
 
                \bibitem{Kec-1994} Alexander S Kechris. { Classical Descriptive Set theory}. 1994, Springer. 

                \bibitem{KorSoi-1998} L. I. Korogodski and Y. S. Soibelman. {Algebras of functions on quantum groups. Part I.} Math. Surveys Monogr. 56, American Mathematical Society, Providence, RI, 1998.


                \bibitem{MatYun-2023} Marco Matassa and Robert Yuncken. {Crystal limits of compact semisimple quantum
groups as higher-rank graph algebras.} J. Reine Angew. Math., 802: 173–221, 2023.

               \bibitem{NesTus-2012} S. Neshveyev and L. Tuset. {Quantized algebras of functions on homogeneous spaces with Poisson stabilizers.} Commun. Math. Phys. 312(1), 223–250 (2012)
\end{thebibliography}
\end{document}